\documentclass{amsart}

\usepackage{amsmath, amssymb}
\usepackage{mathrsfs}   
\usepackage{geometry,graphicx}
\usepackage{url}
\usepackage{caption}
\usepackage{enumerate}

\usepackage{amsthm}
\usepackage[colorlinks=true,linkcolor=blue!50!black,citecolor=blue!50!black,urlcolor=blue!50!black]{hyperref}
\theoremstyle{plain}
\newtheorem{theorem}{Theorem}[section]
\newtheorem{conjecture}[theorem]{Conjecture}
\newtheorem{lemma}[theorem]{Lemma}
\newtheorem{proposition}[theorem]{Proposition}
\newtheorem{corollary}[theorem]{Corollary}

\newtheorem{question}[theorem]{Question}

\theoremstyle{definition}

\theoremstyle{remark}
\newtheorem{remark}[theorem]{Remark}

\usepackage{mathtools}

\newcommand{\R}{\mathbb{R}}

\newcommand{\V}{\mathcal{V}}
\newcommand{\E}{\mathcal{E}}
\newcommand{\conv}{\operatorname{conv}}

\newcommand{\cart}{\mathbin{\square}}
\newcommand{\eps}{\varepsilon}

\DeclareMathOperator{\ints}{int} 
\DeclareMathOperator{\aff}{aff}

\newcommand{\card}[1]{\ensuremath{\# #1}}

\DeclareMathOperator{\Kl}{Kl}
\DeclareMathOperator{\cf}{circ}

\DeclarePairedDelimiter\set{\{}{\}}
\DeclarePairedDelimiter\paren{(}{)}

\usepackage{xcolor}

\usepackage{setspace}
\title[Connectivity, toughness, and hamiltonicity of simplicial
polytopes]
{Connectivity, toughness, and hamiltonicity of simplicial
polytopes}
\author{Guillermo Pineda-Villavicencio}
\address{School of Information Technology, Deakin University, Australia}
\email{\texttt{guillermo.pineda@deakin.edu.au}}
\date{\today}

\begin{document}

\begin{abstract}
For every $d\ge4$, every prescribed connectivity $r\ge d$, and every
$\eps>0$, we construct a simplicial $d$-polytope whose graph has
connectivity exactly $r$, is nontraceable, and has toughness less than
$\eps$. Thus the classical theorems of Whitney (1931) and Tutte (1956), which
guarantee hamiltonicity under sufficiently high connectivity in dimension
three, have no higher-dimensional counterpart. This answers a question posed by Bojan Mohar in 2001 on whether sufficiently high
connectivity guarantees hamiltonicity of $4$-polytopal graphs and the
connectivity questions raised by Ikegami, Maezawa, and Zamfirescu in 2021. 

For fixed $d\ge4$, $r\ge d$, and integer $k\ge2$, the examples can
be chosen to have no closed spanning walk with at most $k$ visits
per vertex, no spanning tree with degrees at most $k$, and
nontraceable prisms.

 For fixed $d$ and $r$, we obtain graphs
on $n$ vertices whose longest paths and cycles contain $O(\sqrt n)$
vertices. Consequently, the family of  simplicial
$d$-polytopes with connectivity exactly $r$ has Gr\"unbaum--Walther shortness coefficient
zero and shortness exponent at most $1/2$.

\end{abstract}

\maketitle

\section{Introduction}
\label{sec:intro}
Throughout, a \emph{$d$-polytope} is a $d$-dimensional convex hull of
finitely many points in $\R^d$, and $G(P)$ denotes the graph (the
$1$-skeleton) of a polytope $P$. A polytope is \emph{simplicial} if all
its facets are simplices. We write $\V(P)$ and $\E(P)$ for the vertex
and edge sets of $P$, and use the same notation for graphs.

For a graph $G=(V,E)$ and $X\subseteq V$, let $G[X]$ denote
the subgraph induced by $X$, and write $G-X:=G[V\setminus X]$.
We prefer $G-x$ to $G-\set{x}$.
A vertex set $Z$ \emph{separates} $G$ if $G-Z$ has at least
two components. Such a set is a \emph{separator}.
A graph $G$ is \emph{$r$-connected} if it has more than $r$
vertices and $G-X$ is connected for every $X\subseteq\V(G)$
with $\card X\le r-1$. Its \emph{connectivity} is the greatest
integer $r$ for which it is $r$-connected.

A \emph{hamiltonian path} in a graph is a path that visits every vertex exactly once. A  graph with a hamiltonian path is \emph{traceable}; otherwise it is \emph{nontraceable}. A \emph{hamiltonian cycle} is a cycle that visits every vertex exactly once. A graph with a hamiltonian cycle is \emph{hamiltonian}; otherwise it is \emph{nonhamiltonian}. For polytopes, these notions apply to their graphs.

Hamiltonian paths and cycles have been widely studied in graph theory. On graphs of polytopes, the focus has mostly been on 3-connected planar graphs, since graphs of 3-polytopes are exactly the planar 3-connected graphs (Steinitz’s theorem \cite{Ste22}).  As examples of such studies we have the celebrated theorems of Whitney \cite{Whi31} and  Tutte \cite{Tut56}.

\begin{theorem}[Whitney  1931, \cite{Whi31}]
    Every 4-connected simplicial 3-polytope is hamiltonian.
    \label{thm:4connected-simplicial-planar-whitney}
\end{theorem}

\begin{theorem}[Tutte 1956, \cite{Tut56}]
    Every 4-connected 3-polytope is hamiltonian.
    \label{thm:4connected-planar-tutte}
\end{theorem}

A question of Mohar~\cite{Moh01} asks for a $4$-dimensional analogue
of Tutte's theorem~\ref{thm:4connected-planar-tutte}.

\begin{question}[Mohar 2001]\label{q:mohar}
Is there a constant $c$ such that every $c$-connected $4$-polytope is
hamiltonian?
\end{question}

Ikegami, Maezawa, and Zamfirescu~\cite[Sec.~3, Item~2]{IkeMaeZam21}
recalled Mohar's question and asked whether there are nonhamiltonian
$5$-connected $4$-polytopes. They also observed that the connectivity
question makes sense in higher dimensions. Accordingly, we consider
the following counterparts of Whitney's and Tutte's theorems in each
fixed dimension at least four.

\begin{question}[Whitney--Tutte type problem for polytopal graphs]\label{q:main}
For every $d\ge 4$, do there exist constants $c_1$ and $c_2$ such that
\begin{enumerate}[\rm (i)]
\item every $c_1$-connected simplicial $d$-polytope is hamiltonian;
\item every $c_2$-connected $d$-polytope is hamiltonian?
\end{enumerate}
\end{question}
We answer both parts of Question~\ref{q:main} in the negative, and
hence also answer Mohar's question. In contrast, the answer to the
$5$-connected existence question of Ikegami, Maezawa, and Zamfirescu
is affirmative. Our examples are simplicial and nontraceable, and
allow the connectivity to be prescribed exactly while the toughness
tends to zero.

We denote by $c(G)$ the number of connected components of a graph $G$.  Following Chv\'atal~\cite{Chv73}, the \emph{toughness} of a noncomplete graph \(G\) is
\[
 \tau(G):=
 \min\set*{
   \frac{\card S}{c(G-S)}:
   S\subseteq\V(G),\ c(G-S)>1}.
\]
The toughness of a complete graph is defined to be infinite. For $t>0$, a graph
is \emph{\(t\)-tough} if \(\tau(G)\geq t\), or equivalently if
\[
 c(G-S)\leq \frac{\card S}{t}
\]
for every separator \(S \) of \(G\). Hamiltonian graphs must be 1-tough (Lemma~\ref{lem:k-walk-tree-toughness}). 

Tutte's theorem~\ref{thm:4connected-planar-tutte} immediately implies that a nonhamiltonian 3-polytope $P$ on at least five vertices contains a separator with three vertices. Thus $\tau(G(P))\le 3/2$. Harant~\cite{Har93} constructed nonhamiltonian  simple 3-polytopes attaining this bound.

Our main result is the following.

\begin{theorem}
\label{thm:main}

Let $d\ge4$, $r\ge d$, and $A\ge3$ be integers, and put
$N=A(r-d+1)$. There is a simplicial $d$-polytope $P$ 
whose vertex set is partitioned as $S\cup T$, with
\[
 \card S=2N+d-4
 \qquad\text{and}\qquad
 \card T=A\bigl((d-3)N-2(d-4)\bigr),
\]
such that
\begin{enumerate}[\rm (i)]
\item $G(P)$ has connectivity exactly $r$; 
\item $T$ is an independent set of $G(P)$ whose vertices all  have
degree $r$ in $G(P)$; and  
\item 
$\tau\paren*{G(P)}\le \frac{\card S}{\card T}\le\frac{2}{A}$.
\end{enumerate}
In particular, $G(P)$ is nontraceable, and hence nonhamiltonian. 
\end{theorem}

\begin{corollary}\label{cor:main}
 For every pair of integers $d\ge4$ and $r\ge d$ and every
 $\eps>0$, there is a nontraceable simplicial $d$-polytope  whose
 graph has connectivity exactly $r$ and toughness less than $\eps$.

Consequently, for every fixed $d\ge4$, neither constant $c_1$ nor
constant $c_2$ in Question~\ref{q:main} exists. In particular, the answer
to Mohar's question~\ref{q:mohar} is negative.
\end{corollary}
\begin{proof}[Proof of Corollary~\ref{cor:main} from Theorem~\ref{thm:main}]
Choose an integer $A>\max\set{2,2/\eps}$ and apply
Theorem~\ref{thm:main}. The resulting graph has toughness
at most $2/A<\eps$. 
For each fixed $d\ge4$, the integer $r$ can be chosen arbitrarily large, so these nontraceable examples answer both parts of
Question~\ref{q:main} in the negative.
\end{proof}

A
\emph{$k$-walk} is a closed spanning walk in which every vertex occurs
at most $k$ times; a hamiltonian cycle is a 1-walk.    A $k$-tree is a spanning tree
with all vertices of degree at most $k$;  a 2-tree is precisely a hamiltonian 
path.  The \emph{prism} over a graph $G$ is the Cartesian product
$G\cart K_2$; it consists of two copies of $G$, with an additional
edge between the two copies of each vertex. We identify its vertex
set with $\V(G)\times\{0,1\}$. A graph is \emph{prism-hamiltonian}
if its prism is hamiltonian. For a polytope $P$, the graph of the
polytope $P\times[0,1]$ is $G(P)\cart K_2$. These notions are related by the following implications
\cite{KaiRyjKra07}:
\begin{equation*}\label{eq:property-implications}
    \text{hamiltonian cycle $\Rightarrow$ hamiltonian path  $\Rightarrow$ hamiltonian prism $\Rightarrow$  $2$-walk $\Rightarrow$ 3-tree}
\end{equation*}

The nontraceability assertion in Theorem~\ref{thm:main} and several
consequences of the construction follow from component-count
obstructions. Assertions (iii) and (iv) below are recorded in \cite[Lem.~2.1]{JW90} and
\cite[Prop.~1]{Oze15}, respectively.

\begin{lemma}[Component-count-toughness  obstructions, {\cite[Lem.~2.1]{JW90},\cite[Prop.~1]{Oze15}}]\label{lem:k-walk-tree-toughness}
Let $G$ be a graph, let $\emptyset\ne X\subseteq\V(G)$, and let
$k\ge1$ be an integer. Then the following hold.
\begin{enumerate}[\rm (i)]
\item If $G$ is hamiltonian, then $c(G-X)\le\card X$.
\item If $G$ is traceable, then $c(G-X)\le\card X+1$.
\item If $G$ has a $k$-walk, then $c(G-X)\le k\cdot \card X$.
\item If a graph \(G\) has a \(k\)-tree, then
\(
 c(G-X)\leq (k-1)\card X+1.
\)
\end{enumerate}
\end{lemma}

Our construction gives a simplicial
$d$-polytope whose graph has connectivity exactly $r$ and has neither
a $k$-walk nor a $k$-tree.

\begin{corollary}\label{cor:no-k-walk-tree}
For every pair of integers $d\ge4$ and $r\ge d$ and every integer
$k\ge2$, there is a simplicial $d$-polytope whose graph has
connectivity exactly $r$, has no $k$-walk, and has no $k$-tree. One may take any integer $A\ge2k+1$
in Theorem~\ref{thm:main}.
\end{corollary}

 \begin{proof}[Proof of Corollary~\ref{cor:no-k-walk-tree} from Theorem~\ref{thm:main}]
Choose $A\ge2k+1$, and let $P$, with partition
$S\cup T$, be given by Theorem~\ref{thm:main}.
Put $G:=G(P)$. By Theorem~\ref{thm:main}(iii),
\[
 c(G-S)=\card T\ge\frac{A}{2}\card S>k\cdot \card S.
\]
Lemma~\ref{lem:k-walk-tree-toughness}(iii) rules out a $k$-walk.
Moreover, since $\card S>1$, we find 
\[
 c(G-S)>k\cdot \card S>1+(k-1)\card S.
\]
Lemma~\ref{lem:k-walk-tree-toughness}(iv) therefore rules out a $k$-tree.
\end{proof}

Thus, even with the dimension and connectivity both fixed, there is
no uniform bound on the number of visits per vertex in a closed
spanning walk or on the maximum degree of a spanning tree.
Barnette's and Gr\"unbaum's constructions \cite[Sec.~5]{Bar66}
already give, for every $k$, a simplicial $d$-polytope with $d\ge4$
that admits no $k$-tree; see also
\cite[Sec.~3, Item~3]{IkeMaeZam21}. However, their graphs have
connectivity exactly $d$. The spanning-tree conclusion in
Corollary~\ref{cor:no-k-walk-tree} strengthens these results by
allowing the connectivity to be prescribed as any integer $r\ge d$
in every fixed dimension $d\ge4$.

Some of our polytopes have nontraceable prisms.

\begin{corollary}\label{cor:nontraceable-prism}
For every pair of integers $d\ge4$ and $r\ge d$, there is a
simplicial $d$-polytope $P$ whose graph has connectivity exactly $r$
and whose prism $G(P)\cart K_2$ is nontraceable. One may take any
integer $A\ge5$ in Theorem~\ref{thm:main}.
\end{corollary}

\begin{proof}[Proof of Corollary~\ref{cor:nontraceable-prism} from Theorem~\ref{thm:main}]
Choose an integer $A\ge5$, and let $P$, with partition
$S\cup T$, be given by Theorem~\ref{thm:main}.
Put $G:=G(P)$, and delete $S\times\{0,1\}$ from
$G\cart K_2$. The remaining graph consists
of $\card T$ disjoint edges, one joining the two copies of each
vertex of $T$. Consequently,
\[
 \begin{aligned}
 c\bigl((G\cart K_2)-(S\times\{0,1\})\bigr)
   &=\card T\ge\frac52\card S>2\card S+1
    =\card{(S\times\{0,1\})}+1,
 \end{aligned}
\]
where the strict inequality follows from $\card S\ge6$.
Lemma~\ref{lem:k-walk-tree-toughness}(ii) proves nontraceability.
\end{proof}

For a graph $H$ containing a cycle, let $\cf(H)$
denote its \emph{circumference}, the length of a longest
cycle. For an infinite family $\mathcal G$ of graphs containing
cycles and having unbounded order, the \emph{shortness coefficient}
and \emph{shortness exponent} of Gr\"unbaum and Walther \cite{GruWal73} are
\[
 \begin{aligned}
 \rho(\mathcal G)
   &:=\liminf_{\substack{H\in\mathcal G\\\card{\V(H)}\to\infty}}
       \frac{\cf(H)}{\card{\V(H)}},\qquad
 \sigma(\mathcal G)
   &:=\liminf_{\substack{H\in\mathcal G\\\card{\V(H)}\to\infty}}
       \frac{\log\cf(H)}{\log\card{\V(H)}}.
 \end{aligned}
\]

\begin{corollary}\label{cor:shortness}
Fix integers $d\ge4$ and $r\ge d$, and let $\mathcal S_{d,r}$ be the
family of graphs of simplicial $d$-polytopes with connectivity
exactly $r$.
\begin{enumerate}[{\rm (i)}]
\item For a graph $G$ obtained from Theorem~\ref{thm:main}, every
cycle has at most $2\card S$ vertices and every path has at most
$2\card S+1$ vertices.
\item We have
\(
 \rho(\mathcal S_{d,r})=0\) and \(\sigma(\mathcal S_{d,r})\le\frac12
\). 
More precisely, $\mathcal S_{d,r}$ contains graphs of unbounded order $n$
whose longest  paths and cycles have $O(\sqrt n)$ vertices.
\end{enumerate}
\end{corollary}

\begin{proof}
(i) Since $T$ is independent, a cycle contains at most as many
vertices of $T$ as of $S$, and a path contains at most one more.
This gives the asserted vertex bounds.

(ii) For each integer $A\ge3$, choose a graph $G_A$ from 
Theorem~\ref{thm:main}, with partition $S_A\cup T_A$, and put
$n_A:=\card{\V(G_A)}$. The theorem gives
\[
 \begin{aligned}
 n_A&=(d-3)(r-d+1)A^2+O(A)\; \text{and}\;
 \card{S_A}&=2(r-d+1)A+d-4.
 \end{aligned}
\]
Thus $n_A\to\infty$ as $A\to\infty$. Part~(i) now ensures \(\cf(G_A)\le C\sqrt{n_A}\), where  \(C>0\) depends only on \(d\) and \(r\).  Consequently, \[
 \frac{\log\cf(G_A)}{\log n_A}
 \le\frac12+\frac{\log C}{\log n_A},
 \qquad
 \frac{\cf(G_A)}{n_A}\le\frac{C}{\sqrt{n_A}}\longrightarrow0.
\]
Since $G_A\in\mathcal S_{d,r}$, we find 
\[
 \begin{aligned}
 \sigma(\mathcal S_{d,r})
 &\le\liminf_{A\to\infty}
     \frac{\log\cf(G_A)}{\log n_A}\le\frac12\quad \text{and}\quad 
 0\le\rho(\mathcal S_{d,r})
 &\le\liminf_{A\to\infty}
     \frac{\cf(G_A)}{n_A}=0.
 \end{aligned}
\]
This completes the proof.
\end{proof}


\subsection{Existing nonhamiltonicity results for higher-dimensional polytopes}

In dimensions at least four, nontraceability, and hence
nonhamiltonicity, is classical. Gr\"unbaum and Motzkin~\cite{GruMot62}
constructed families of $d$-polytopes on $n$ vertices, for every fixed
$d\ge3$, whose longest paths have sublinear length
$O(n^{1-2^{-19}})$. Moon and Moser~\cite{MooMos63} improved this bound
to $O(n^{\log_d2})$ using simplicial $d$-polytopes. The Moon--Moser
construction uses successive Kleetopes of a simplex; Kleetopes are defined in Section~\ref{sec:preliminaries}. 
Thus being simplicial alone guarantees neither traceability nor
hamiltonicity.

The elementary Kleetope obstruction is also described by Klee in
Gr\"unbaum's book~\cite[Sec.~17.1]{Gru03}. If a simplicial polytope has
$v$ vertices and $f>v$ facets, then its Kleetope is nonhamiltonian:
deleting the original vertices leaves $f$ isolated vertices.
The same deletion proves nontraceability when $f>v+1$ and gives
toughness at most $v/f$. 

Another source of nonhamiltonian $4$-polytopes is the prism
construction.
\v{S}pacapan~\cite{Spa21} disproved the Rosenfeld--Barnette conjecture
that every $3$-connected planar graph has a hamiltonian prism.
By Steinitz's theorem, his counterexample yields a $3$-polytope whose
prism is a nonhamiltonian $4$-polytope.
\v{S}pacapan also noted that a slight modification of his proof shows
that sufficiently large members of his family have nontraceable prisms.
Extending his construction, Ikegami, Maezawa, and
Zamfirescu~\cite{IkeMaeZam21} obtained an infinite family of
nonhamiltonian $4$-polytopes in which the proportion of vertices
of degree four tends to one.

In the Moon--Moser construction, the vertices added at the last stage
have degree $d$, so the resulting graphs have connectivity exactly $d$.
The prisms constructed by \v{S}pacapan and by Ikegami, Maezawa, and
Zamfirescu have connectivity exactly four.
These examples therefore do not yield arbitrarily high connectivity
in a fixed dimension. Corollary~\ref{cor:nontraceable-prism} gives nontraceable prisms over
simplicial polytopes in every fixed dimension $d\ge4$, with the
graphs of the base polytopes having any prescribed connectivity
$r\ge d$.

The Moon--Moser bound already implies shortness coefficient zero
and shortness exponent at most $\log_d2$ for the family of all
graphs of simplicial $d$-polytopes. Corollary~\ref{cor:shortness}
gives coefficient zero and exponent at most $1/2$ for each family
$\mathcal S_{d,r}$ by constructing graphs whose longest paths and
cycles have $O(\sqrt n)$ vertices. The exponent $1/2$ equals the Moon--Moser
bound when $d=4$ and is larger when $d>4$. The contribution here is the additional requirement
of having connectivity exactly a prescribed $r\ge d$.

There is a related distinction for spanning trees.
Barnette~\cite[Thm.~1]{Bar66} proved that every 3-polytope admits a 3-tree. He also constructed,  for each $k$, a $5$-polytope admitting
no $k$-tree~\cite[Thm.~2]{Bar66},  and recorded Gr\"unbaum's analogous
$4$-dimensional construction~\cite[Sec.~5]{Bar66}; see also
\cite[Sec.~3, Item~3]{IkeMaeZam21}. Both constructions use
Kleetopes of cyclic polytopes and are therefore simplicial.
Their added vertices have degree $d$, so their graphs have
connectivity exactly $d$. Corollary~\ref{cor:no-k-walk-tree}
extends this phenomenon to every fixed dimension $d\ge4$ and
every prescribed connectivity $r\ge d$, ruling out both $k$-trees
and $k$-walks.

A polytope is \emph{simple} if its dual is simplicial.
Barnette's conjecture for simple $4$-polytopes%
~\cite[Conj.~1.5]{Gru70} has the following higher-dimensional form.

\begin{conjecture}[Barnette's conjecture, higher-dimensional form]
\label{conj:Barnette-simple-hamiltonian}
Every simple $d$-polytope with $d\ge4$ is hamiltonian.
\end{conjecture}

The duals of our polytopes are hamiltonian by
Proposition~\ref{prop:dual-hamiltonian}, so this construction
yields no counterexamples to
Conjecture~\ref{conj:Barnette-simple-hamiltonian}.

Section~\ref{sec:preliminaries} gives the preliminaries,
Section~\ref{sec:main} the construction and proof of
Theorem~\ref{thm:main}, and Appendix~\ref{app:dual-hamiltonian}
the proof of Proposition~\ref{prop:dual-hamiltonian}.

\section{Preliminaries}
\label{sec:preliminaries}

For a polytope $P$, write $f_j(P)$ for the number of its $j$-faces.
We first record the properties of direct sums of polytopes that we use.
Let
\(U\subset \R^d\) and \(W\subset \R^e\) be a $d$-polytope and an $e$-polytope, respectively, with $d,e\ge1$, translated so that
the origin lies in their relative interiors. The \emph{direct sum} \(U\oplus W\) of \(U\) and \(W\) is the \((d+e)\)-polytope \[
 U\oplus W
 :=\conv\paren*{
      \paren*{U\times\set{0}}
      \cup
      \paren*{\set{0}\times W}
    }
 \subseteq\R^{d+e}.
\]
Here $\conv X$ denotes the convex hull of $X$.  We gather some basic properties of a direct sum \cite[Sec.~2.6]{Pin24}.

\begin{proposition} Let
\(U\subset \R^d\) and \(W\subset \R^e\) be a $d$-polytope and an $e$-polytope, respectively, with $d,e\ge1$, translated so that
the origin lies in their relative interiors. 
    \label{prop:direct-sum-basic}

    \begin{enumerate}[{\rm (i)}]
\item The proper faces of $U\oplus W$ are precisely the sets
\[
 \conv\bigl((F\times\{0\})\cup(\{0\}\times F')\bigr),
\]
where $F$ and $F'$ are faces of $U$ and $W$, respectively,
with $F\ne U$ and $F'\ne W$, and at most one of them is empty. 
Its facets are precisely the joins of a facet of $U$ and a
facet of $W$. In particular, the direct sum of two simplicial
polytopes is simplicial.

\item In \(U\oplus W\)  every vertex of \(U\) is adjacent to every vertex of
\(W\). Moreover, if \(\dim U\geq2\), then the subgraph of
\(G(U\oplus W)\) induced by \(\V(U)\) is \(G(U)\).
    \end{enumerate}
\end{proposition}

For connectivity arguments we use Balinski's theorem.

\begin{theorem}[Balinski~\cite{Bal61}]\label{thm:balinski}
The graph of a $d$-polytope is $d$-connected.
\end{theorem}

\subsection{Stellar subdivisions of edges}

Following \cite[Section~2.10]{Pin24}, let $Q\subset\R^d$ be a $d$-polytope
and let $F$ be a facet of $Q$. Denote by $\ints Q$ the interior of $Q$ and by $\aff F$ the affine hull of $F$. The facet $F$ is \emph{visible from} a point
$p$ with respect to $Q$ if $p$ belongs to the open halfspace determined by
$\aff F$ that is disjoint from $Q$. If instead $p$ belongs to the open
halfspace determined by $\aff F$ that contains $\ints Q$, then $F$ is
\emph{nonvisible from} $p$. Moreover, the point $p$ is \emph{beyond} a face $J$ of $Q$ if the facets of $Q$ containing $J$ are precisely those that are visible from $p$ and the facets of $Q$ not containing $J$ are all nonvisible from $p$. Finally,
$p$ is in \emph{general position} with respect to $Q$ if it lies on no facet
hyperplane of $Q$.

We need \cite[Theorem~2.10.1]{Pin24}, which can also  be found in \cite[Thm.~5.2.1]{Gru03}.

\begin{theorem}[Inductive construction of polytopes, {\cite[Theorem~2.10.1]{Pin24}}] 
\label{thm:construction} Let $Q$ and $Q'$ be two $d$-polytopes in $\mathbb{R}^d$, and let $v'$ be a vertex of $Q'$ such that $v'\not\in Q$ and $Q'=\conv (Q\cup \{v'\})$. Then the following hold:
\begin{enumerate}[{\rm (i)}]
\item A face $F$ of $Q$ is a face of $Q'$ if and only if there exists a facet  of $Q$ containing $F$ that is nonvisible from $v'$ with respect to $Q$.

\item  If  $F$ is a face of $Q$ with $v'\in \aff F$, then $F':=\conv(F\cup\{v'\})$ is a face of $Q'$.

\item   If  $F$ is a face of $Q$ such that, among the facets of $Q$ containing $F$, there is at least one that is visible from $v' $ (with respect to $Q$) and at least one that is nonvisible (with respect to $Q$), then \begin{equation*}
	F':=\conv(F\cup\{v'\})
\end{equation*} 
 is a face of $Q'$. 

\item For each face $F'$ of $Q'$, there is a face $F$ of $Q$ for which (i), (ii), or (iii) applies. In other words, each face of $Q'$ falls precisely in one of the above cases. 
\end{enumerate}
  
\end{theorem} 

We require the stellar subdivision of an edge, which is the edge case
of the geometric stellar subdivision of Ewald and
Shephard~\cite[p.~10, Statement~(4)]{EwaShe74}. Let $J$ be a proper face of a simplicial $d$-polytope
$Q\subset \R^d$, and let $z\in \R^d$ be a point beyond $J$. Then we say that the 
polytope $Q':=\conv(Q\cup\{z\})$ is obtained from $Q$ by the \emph{stellar subdivision  of $Q$ with respect to $J$}. The next lemma gives the precise facet and graph
descriptions that we need.

\begin{lemma}[Stellar subdivision of an edge]\label{lem:stellar-edge}
Let $d\ge2$, let $Q\subset \R^d$ be a simplicial $d$-polytope, and let $[u,v]$ be an edge  of $Q$. Suppose that $z\in \R^d$ is a point beyond $[u,v]$ and $Q':=\conv(Q\cup\{z\})\subset \R^d$ is the stellar subdivision of  $Q$   with respect to $[u,v]$. Then the following hold. 
\begin{enumerate}[{\rm (i)}]
\item The
polytope $Q'$ has vertex set
$\V(Q)\cup\{z\}$. 
\item The facets of $Q'$ are all simplices obtained by retaining every facet $F$ of
$Q$ not containing $[u,v]$ and replacing each facet $F$ containing
$[u,v]$ by
\[
 \conv\bigl((\V(F)\setminus\{v\})\cup\{z\}\bigr)
 \quad\text{and}\quad
 \conv\bigl((\V(F)\setminus\{u\})\cup\{z\}\bigr).
\]

\item The subgraph of $G(Q')$ induced by $\V(Q)$ is obtained from $G(Q)$ by
deleting the edge $uv$.   The neighbours of $z$  are exactly
the vertices that lie in a facet of $Q$ 
containing $[u,v]$.
\end{enumerate}
\end{lemma}

\begin{proof}
Translate $Q$ so that $0\in\ints Q$ and take
$x$ in the relative interior of $[u,v]$. For sufficiently small $\eps>0$, the point
$z_0:=(1+\eps)x$ lies beyond $[u,v]$, so such points exist.
Now let $z$ be any point beyond $[u,v]$. By definition, $z$
is in general position with respect to $Q$, and its visible
facets are precisely those containing $[u,v]$.

(ii) A facet of $Q$ not containing $[u,v]$ is nonvisible from $z$, so it is a facet of $Q'$ by Theorem~\ref{thm:construction}(i). Let $F$ be a facet of $Q$ containing $[u,v]$. A ridge of  $F$ containing $[u,v]$  lies in two facets visible  from $z$, so it does not give rise to a facet in $Q'$ (Theorem~\ref{thm:construction}). Each of the other two
ridges of $F$, obtained by omitting $u$ or $v$ from $\V(F)$, is contained in a second facet that does not contain $[u,v]$ and is therefore nonvisible from $z$. Consequently, it gives rise to a facet of $Q'$ by Theorem~\ref{thm:construction}(iii). By Theorem~\ref{thm:construction}(iv), these are all the facets
of $Q'$, all of which are simplices.

(i)--(iii) Every vertex of $Q$ and every edge other than
$[u,v]$ lie in a facet not containing $[u,v]$. These vertices and edges therefore remain
faces of $Q'$ by Theorem~\ref{thm:construction}(i). All facets containing $[u,v]$ are visible, so $[u,v]$ is
not a face of $Q'$. The facet description in (ii) shows that
$z$ is a vertex and that its neighbours are precisely the
vertices of the original facets containing $[u,v]$.
Thus $\V(Q')=\V(Q)\cup\{z\}$, and the asserted graph
description follows.
\end{proof}

\begin{remark}[Repeated stellar subdivisions]\label{rmk:repeated-stellar-subdivisions}
Let $Q$ be a simplicial $d$-polytope, with $d\ge2$, and let
$f:=[u,v]$ be an edge of $Q$.
Let $F_1,\ldots,F_s$ be all the facets of $Q$ containing $f$, 
and put $B_{F_j}:=\V(F_j)\setminus\{u,v\}$ for $1\le j\le s$.
For any integer $\ell\ge1$, successive applications of
Lemma~\ref{lem:stellar-edge} replace $f$ by an induced path
$L_f$ of length $\ell$, with vertices in order
\[
 u=x_0,x_1,\ldots,x_\ell=v.
\]
Let $Q'$ denote the resulting polytope. 

Every original facet not
containing $[u,v]$ is retained in $Q'$. For each $1\le i\le\ell$,
the facets of $Q'$ containing $[x_{i-1},x_i]$ are precisely
\[
 \conv\bigl(B_{F_j}\cup\{x_{i-1},x_i\}\bigr),
 \qquad 1\le j\le s.
\]
Together with the retained facets, these are all the facets of $Q'$. See Figure~\ref{fig:repeated-stellar-subdivisions}. This follows by induction on the number of subdivisions: subdividing
an edge $[x,y]$ of the current path by a new vertex $z$ replaces each
facet with vertex set $B_{F_j}\cup\{x,y\}$ by two facets with vertex
sets $B_{F_j}\cup\{x,z\}$ and $B_{F_j}\cup\{z,y\}$.
Thus the same complementary sets $B_{F_j}$ occur for every edge
of the path; see Figure~\ref{fig:repeated-stellar-subdivisions}(b)--(c). 

Each new vertex $x_i$, $1\le i\le\ell-1$, is
adjacent, outside $L_f$, to exactly the vertices in
$\bigcup_{j=1}^s B_{F_j}$. 

The facet description also shows that
no chords of $L_f$ are introduced.
\end{remark}

\begin{figure}[htbp] 
\begin{center}     
    \includegraphics[scale=0.8]{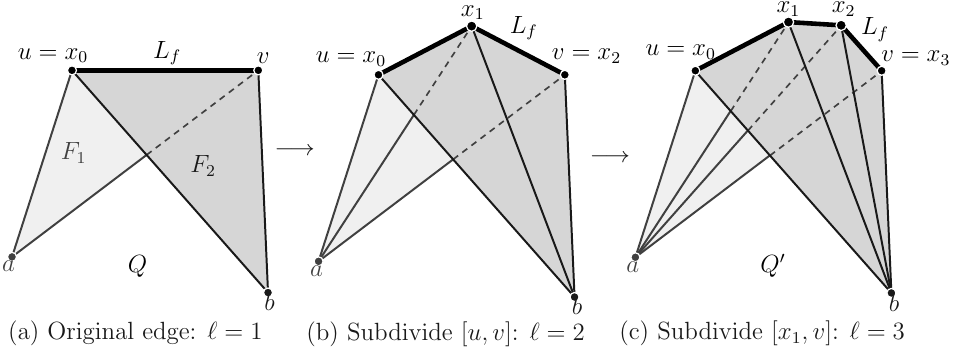}
\end{center}
\caption{Schematic of repeated stellar subdivisions of an edge $f:=[u,v]$ in a simplicial 3-polytope $Q$.  The path $L_f$ is highlighted in bold.  (a)  The edge $f$ is contained in the facets $F_1,F_2$ with $B_{F_1}=\set{a}$  and $B_{F_2}=\set{b}$. (b)--(c)  The edge $[x_{i-1},x_i]$ of $L_f$ lies in the facets  $\conv(B_{F_j}\cup\{x_{i-1},x_i\})$, $j=1,2$.  The same sets $B_{F_1}$ and $B_{F_2}$ occur for every edge of $L_f$.}
\label{fig:repeated-stellar-subdivisions}
\end{figure}

For a  $d$-polytope ($d\ge 2$), \emph{stacking over a facet}  means
adding a point beyond it and taking the convex hull. Successive stackings from a simplex produce a \emph{stacked}
simplicial polytope.   
Stacking once over each original facet of a polytope $Q$
produces its \emph{Kleetope} $\Kl(Q)$. 

\begin{proposition}\label{prop:kleetopes} Let $Q$ be a $d$-polytope with $d\ge 2$. The following hold.
\begin{enumerate}[{\rm (i)}]
    \item If $Q$ is simplicial, then so is its Kleetope  $\Kl(Q)$. 
    \item The added vertices form an independent
set of $G(\Kl(Q))$, and each has precisely the vertices of its corresponding original
facet as neighbours. 
\item If $\dim Q\ge3$, every original edge of $Q$ is
retained in $\Kl(Q)$.
\end{enumerate}
\end{proposition} 

\section{Proof of Theorem~\ref{thm:main}}
\label{sec:main}

Inspired by the Kleetope obstruction in \cite[Sec.~17.1]{Gru03},
we form the Kleetope of a direct sum and subdivide edges of one
factor to increase the added vertices' degrees while preserving
their independence. The remaining vertices induce the join of
the graph of a stacked $(d-2)$-polytope and a cycle, simplifying
the connectivity proof.

\textbf{The construction.}
Fix $d\ge4$, $r\ge d$, and $A\ge3$, and put $N:=A(r-d+1)$, as in the
statement. In particular, $N\ge3(r-d+1)$. Let $R$ be a stacked simplicial
$(d-2)$-polytope obtained from a $(d-2)$-simplex by $N-3$ stackings.
Each stacking adds one vertex and replaces one facet by $d-2$ facets.
Consequently,
\begin{equation}\label{eq:stacked-counts}
 \begin{aligned}
  \card{\V(R)}&=(d-1)+(N-3)=N+d-4,\\
  f_{d-3}(R)&=(d-1)+(N-3)(d-3)
                  =(d-3)N-2(d-4).
 \end{aligned}
\end{equation}
For $d=4$, the polytope $R$ is simply an $N$-gon.

Let $D$ be an $A$-gon. Translate $R$ and $D$ so that the origin lies in
their respective interiors, and form
\[
 Q:=R\oplus D,\qquad K:=\Kl(Q).
\]
We identify $R$ with $R\times\{0\}$ and $D$ with
$\{0\}\times D$ in $\R^{d-2}\times\R^2$. By
Proposition~\ref{prop:direct-sum-basic}, $Q$ is a simplicial
$d$-polytope whose facets are precisely
\[
 F*f=\conv(F\cup f),
 \qquad F\text{ a facet of }R,\quad f\in\E(D).
\]
Here $F*f$ denotes the free join of $F$ and $f$ \cite[Sec.~2.6]{Pin24}. 
For every such facet $F*f$, let $p_{F,f}$ be the vertex added in forming $K$,
and put
\[
 T:=\set*{p_{F,f}:F\text{ a facet of }R,\ f\in\E(D)}.
\]
By Proposition~\ref{prop:kleetopes}, $K$ is simplicial (since $Q$ is simplicial), retains
every edge of $Q$ (since $\dim Q\ge 4$), and has $T$ as an independent set. Each $p_{F,f}$
has precisely the vertices of $F*f$ as neighbours. Moreover,
\begin{equation}\label{eq:T-count}
 \card T=A f_{d-3}(R)=A\bigl((d-3)N-2(d-4)\bigr).
\end{equation}

For each edge $f$ of $D$, regarded as an edge of $K$,
use successive stellar subdivisions to replace $f$ by
a path $L_f$ of $r-d+1$ edges; see Lemma~\ref{lem:stellar-edge} and Remark~\ref{rmk:repeated-stellar-subdivisions}.  This introduces $r-d$
vertices for each $f$. When $r=d$, no subdivisions are made and $L_f$
is the original edge. Let $P$ be the resulting simplicial $d$-polytope.

Put $U:=\V(R)$, let $W$ consist of the original vertices of $D$ and
all subdivision vertices, and set $S:=U\cup W$. By
Lemma~\ref{lem:stellar-edge}, after each subdivision, every old vertex is retained, so
\[
 \V(P)=S\cup T,\qquad S\cap T=\emptyset.
\]

\textbf{The effect of the subdivisions.}
Fix a facet $F$ of $R$ and an edge $f=[u,v]$ of $D$. Write the vertices of $L_f$ in order as
$u=x_0,x_1,\ldots,x_\ell=v$, where $\ell=r-d+1$. 
The facets of $K$ arising from $F*f$ are
\begin{equation}\label{eq:kleetope-facets}
 \begin{gathered}
  \conv\paren*{\V(F)\cup\set*{p_{F,f},u}},\qquad
  \conv\paren*{\V(F)\cup\set*{p_{F,f},v}},\\
  \conv\paren*{(\V(F)\setminus\set*{w})
                 \cup\set*{p_{F,f},u,v}},
       \quad \text{for each $w\in\V(F)$}.
 \end{gathered}
\end{equation}
In particular, each facet of $K$ contains at most one edge of $D$.
Subdividing $f$ changes only facets containing $f$, and no replacement
facet contains another edge of $D$. Each facet produced by further subdivisions of the edges
replacing $f$ contains at most one vertex of $D$; any such
vertex is $u$ or $v$. Thus the facets containing any other edge of $D$ are unchanged,
and subdivisions associated with distinct edges of $D$ do not affect
one another.

The first two facets in~\eqref{eq:kleetope-facets} remain facets of
$P$. By Remark~\ref{rmk:repeated-stellar-subdivisions}, the remaining facets in that display are
replaced by
\begin{equation}\label{eq:subdivided-facets}
\conv\bigl((\V(F)\setminus\{w\})
                 \cup\{p_{F,f},x_{i-1},x_i\}\bigr),
 \qquad w\in\V(F),\quad 1\le i\le \ell.
\end{equation}
Together with the first two facets in~\eqref{eq:kleetope-facets},
these are all the facets of $P$ as $F$ and $f$ vary. See Figure~\ref{fig:P-construction} for an example with $d=4$ and $r=5$.

\begin{figure}[htbp] 
\begin{center}     
    \includegraphics[scale=0.8]{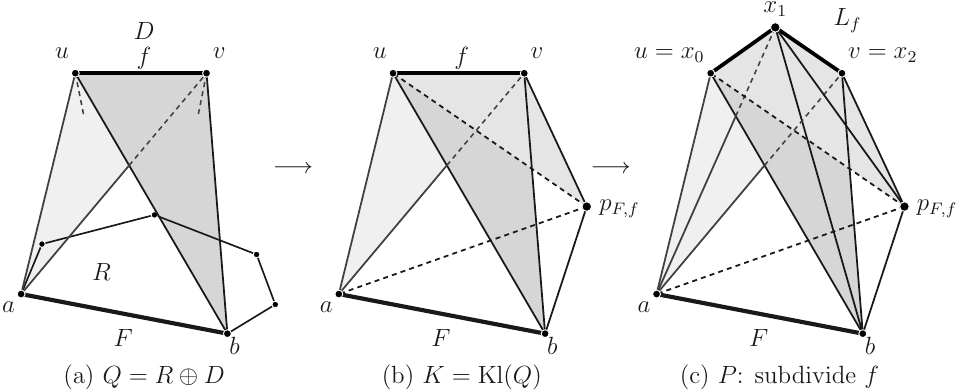}
\end{center}
\caption{Schematic of the construction for $d=4$ and
$r=5$: $R$ is a polygon, and only one edge of $D$ is depicted.
(a) The direct sum $Q=R\oplus D$, with $F=[a,b]$ and $f=[u,v]$.
(b) Stacking over $F*f$ introduces the vertex $p_{F,f}$.
(c) Subdividing $f$ replaces it by the path
$L_f:u=x_0,x_1,x_2=v$.
Only the part associated with $F*f$ is shown
in (b) and (c).}
\label{fig:P-construction}
\end{figure}

It follows that the neighbours of $p_{F,f}$ in $G(P)$ are precisely
the vertices in $\V(F)\cup\V(L_f)$. These two sets are disjoint, so
\begin{equation}\label{eq:degree-T}
 \deg_{G(P)}(p_{F,f})
   =(d-2)+\card{\V(L_f)}
   =(d-2)+(r-d+2)=r.
\end{equation}
No listed facet contains two vertices of $T$, so $T$ remains
independent, and Part (ii) is settled.

\textbf{The graph induced by $S$}. The graph induced by $U$
is $G(R)$: it has this form in $Q$ by
Proposition~\ref{prop:direct-sum-basic}, and neither forming $K$ nor
subdividing the edges in question changes any edge with both ends in
$U$. The graph induced by $W$ is the cycle obtained from the boundary
of $D$ by replacing each edge $f$ by $L_f$. Indeed, no chords of any
$L_f$ are introduced, and the facet descriptions above give no further
edges between vertices of $W$; see Remark~\ref{rmk:repeated-stellar-subdivisions}.

Every original edge between $U$ and $\V(D)$ survives. Also, every
subdivision vertex $x_i$ on $L_f$ is adjacent to every vertex $a\in U$.
To see this, choose a facet $F$ of $R$ containing $a$. The inequality
$d-2\ge2$ gives a vertex $w\in\V(F)\setminus\{a\}$. For each
vertex $x_i$ on $L_f$, a facet in~\eqref{eq:subdivided-facets}
with this choice of $w$ contains both $x_i$ and $a$. Because \(P\) is simplicial, they are adjacent. 
Consequently, $G(P)[S]$ consists of $G(R)$ and a cycle, with all
possible edges between their vertex sets. The counts are
\begin{equation}\label{eq:S-count}
 \begin{aligned}
  \card U&=N+d-4,\quad
  \card W=A+A(r-d)=N,\quad
  \card S&=2N+d-4.
 \end{aligned}
\end{equation}
Combining~\eqref{eq:T-count} and~\eqref{eq:S-count} gives
\[
 \card{\V(P)}
  =2N+d-4+A\bigl((d-3)N-2(d-4)\bigr)
  =\bigl(A(d-3)+2\bigr)N-(2A-1)(d-4).
\]

(i) \textbf{Exact $r$-connectivity.}
Write $G:=G(P)$. Our parameter choices give
\begin{equation}\label{eq:core-size-bounds}
 \begin{aligned}
  \card U&=N+d-4
     \ge3(r-d+1)+d-4
      =r-1+2(r-d)\ge r-1,\\
  \card W&=N=A(r-d+1)\ge 3(r-d+1)\ge r-d+3.
 \end{aligned}
\end{equation}
In particular, 
\begin{equation} \label{eq:G-size-bounds}
    \card{\V(G)}\ge\card S\ge(r-1)+(r-d+3)
=2r-d+2\ge r+2>r.
\end{equation}

Let $X\subseteq\V(G)$ satisfy $\card X\le r-1$. We  prove that
$G[S\setminus X]$ is nonempty and connected. If both $U\setminus X$ and $W\setminus X$ are nonempty, then every
vertex of one set is adjacent to every vertex of the other, so
$G[S\setminus X]$ is connected. If $U\subseteq X$, then
$\card U\ge r-1$ in \eqref{eq:core-size-bounds} forces $X=U$, and the cycle on $W$ is intact.
Finally, suppose that $W\subseteq X$. At most
\[
 \card X-\card W\le r-1-(r-d+3)\le d-4
\]
vertices of $R$ have been deleted. The graph $G(R)$ is
$(d-2)$-connected by Theorem~\ref{thm:balinski}, so
$G[U\setminus X]$ is nonempty and connected.
These cases exhaust all possibilities, because
$\card S>r-1\ge\card X$.

Every vertex of $T\setminus X$ has $r$ neighbours, all in $S$, by
\eqref{eq:degree-T}. At most $r-1$ of these neighbours belong to $X$,
so it has a neighbour in $S\setminus X$. Hence $G-X$ is connected,
proving that $G$ is $r$-connected. Conversely, deleting the $r$ neighbours of any vertex of $T$
isolates it. Since $\card{\V(G)}\ge r+2$ (see \eqref{eq:G-size-bounds}), another vertex remains,
so the resulting graph is disconnected.  Thus $G$ has 
connectivity exactly $r$.

(iii) \textbf{Nontraceability and toughness.}
By~\eqref{eq:stacked-counts},
\[
 f_{d-3}(R)-\card U=(d-4)(N-3)\ge0.
\]
Also $\card W=N\le\card U$, so
\[
 \card S=\card U+\card W\le2\card U\le2f_{d-3}(R),\qquad
 \card T=A f_{d-3}(R).
\]
Deleting $S$ leaves precisely the vertices of $T$ as isolated vertices.
It follows that
\[
 c(G-S)-\card S
  =\card T-\card S
  \ge(A-2)f_{d-3}(R)\ge f_{d-3}(R)\ge d-1\ge 3.
\]
Lemma~\ref{lem:k-walk-tree-toughness} implies that $G$ is nontraceable,
and hence nonhamiltonian. Finally, $c(G-S)=\card T>1$, so
\[
 \tau(G)\le\frac{\card S}{c(G-S)}=\frac{\card S}{\card T}\le \frac{2f_{d-3}(R)}{Af_{d-3}(R)}=\frac{2}{A}.
\]
This completes the proof of the theorem. 

\section{Declaration on the use of generative AI and AI-assisted technologies}
\label{subsec:acknowledments-ai} 

The author used OpenAI's GPT-6 Astra~\cite{OpenAIChatGPT}
to explore research strategies, assist with proofreading,
refine the presentation of proofs, and generate initial figures.
He critically reviewed and, where needed, edited all AI-assisted
material, independently verified the mathematics and figures,
and takes full responsibility for the manuscript's content
and correctness.

\appendix
\section{Hamiltonicity of the dual graphs}
\label{app:dual-hamiltonian}

\begin{proposition}\label{prop:dual-hamiltonian}
Every polytope constructed in the proof of
Theorem~\ref{thm:main} has a hamiltonian dual graph.
\end{proposition}

\begin{proof}
Let $P$ be a $d$-polytope constructed in Theorem~\ref{thm:main}.
Put $\ell:=r-d+1\ge1$. First, $G(R^*)$ has a hamiltonian cycle.
Since $R$ is stacked, $R^*$ is obtained from a $(d-2)$-simplex
by successive vertex truncations
\cite[Sec.~2.6, Connected sums]{Pin24}. Start with a hamiltonian
cycle of the simplex. At each truncation, replace the truncated
vertex by a spanning path in the new simplex facet between the
new vertices on the two incident edges of the current cycle.

Orient the boundary of $D$ cyclically, and write $t(f)$ and $h(f)$
for the tail and head of each oriented edge $f$, respectively.
For each pair $(F,f)$, let $H_{F,f}$ be the graph of the facet
of $P^*$ conjugate to $p_{F,f}$ of $P$. The first two facets
in~\eqref{eq:kleetope-facets} remain facets of $P$ containing
$p_{F,f}$. For $f=[u,v]$ oriented from $u$ to $v$, we have
$t(f)=u$ and $h(f)=v$. Write $t_{F,f}$ and $h_{F,f}$ for the
dual vertices conjugate, respectively, to the facets
with vertex sets
\[\V(F)\cup\set*{p_{F,f},t(f)},\qquad
\V(F)\cup\set*{p_{F,f},h(f)}.\] For the edge $f$, write
\[
 L_f:=x_0x_1\cdots x_\ell,
 \qquad x_0=t(f),\quad x_\ell=h(f),
\]
for the path in $P$ replacing $f$.
For the \(i\)-th edge \(x_{i-1}x_i\),
\eqref{eq:subdivided-facets} gives the facet
\[
 M_{i,w}:=\conv\bigl((\V(F)\setminus\{w\})
 \cup\{p_{F,f},x_{i-1},x_i\}\bigr),
 \qquad w\in\V(F).
\]
Since \(F\) has \(d-2\) vertices, there are \(d-2\) such facets,
one for each choice of the omitted vertex \(w\).
For each $1\le i\le\ell$, let $J_i$ be the subgraph of
$H_{F,f}$ induced by the dual vertices corresponding to
$M_{i,w}$, where $w\in\V(F)$.
These are precisely the vertices of the $(d-3)$-face of
$P^*$ conjugate to
$\conv\{p_{F,f},x_{i-1},x_i\}$.
This face has $d-2$ vertices and is therefore a simplex,
so $J_i$ is a $(d-2)$-clique.
The cliques $J_1,\ldots,J_\ell$ are ordered along $L_f$
and together contain all vertices of $H_{F,f}$ except
$t_{F,f}$ and $h_{F,f}$.

For $1\le i<\ell$, the edges between $J_i$ and $J_{i+1}$
form a perfect matching pairing vertices with the same
omitted vertex of $F$. Indeed,
\[
 M_{i,w}\cap M_{i+1,w}
 =\conv\bigl((\V(F)\setminus\{w\})
       \cup\{p_{F,f},x_i\}\bigr).
\]
This intersection has $(d-3)+2=d-1$ vertices and is therefore
a ridge of $P$, so the corresponding dual vertices are adjacent.
If $w\ne w'$, the facets $M_{i,w}$ and $M_{i+1,w'}$ share only
$d-2$ vertices, so their dual vertices are not adjacent.

Moreover, from \eqref{eq:kleetope-facets} and \eqref{eq:subdivided-facets} it follows that $t_{F,f}h_{F,f}$ is an edge of $H_{F,f}$,
the vertex $t_{F,f}$ is adjacent to every vertex of $J_1$,
and the vertex $h_{F,f}$ is adjacent to every vertex of $J_\ell$.

For any distinct $x,y\in\V(J_1)$, the graph $H_{F,f}$ has
a spanning $xy$-path containing $t_{F,f}h_{F,f}$.
To obtain one, traverse $J_1-y$ starting at $x$, then traverse
$J_2,\ldots,J_\ell$ in order, using a matching edge between
successive cliques, and finish with $h_{F,f},t_{F,f},y$.
The first traversal is possible since $d-2\ge2$,
and each later clique can be traversed starting at the matching
neighbour of the preceding endpoint.
This also applies when $\ell=1$. See Figure~\ref{fig:dual-hamiltonian-construction}(a). 

Consider a hamiltonian cycle $C$ of $G(R^*)$ and fix an edge
$f$ of $D$. For each facet $F$ of $R$, let $F^-$ and $F^+$
be its predecessor and successor on $C$. Since $F$ is a simplex,
there are distinct vertices $w_-,w_+\in\V(F)$ such that
\[
 \V(F\cap F^-)=\V(F)\setminus\{w_-\},
 \qquad
 \V(F\cap F^+)=\V(F)\setminus\{w_+\}.
\]
The dual vertices conjugate to $M_{1,w_-}$ and $M_{1,w_+}$
are therefore distinct vertices of the first clique $J_1$
of $H_{F,f}$, and thus are adjacent. Orient the spanning path in $H_{F,f}$ described
above from the dual vertex corresponding to $M_{1,w_-}$
to that corresponding to $M_{1,w_+}$.

For this fixed $f$, the chosen spanning paths in the graphs
$H_{F,f}$, one for each facet $F$ of $R$, are vertex-disjoint,
since every facet in the block $H_{F,f}$ contains $p_{F,f}$
and no other stacking vertex.
If $F'$ is the successor of $F$ on $C$, the terminal vertex
of the path in $H_{F,f}$ and the initial vertex of the path in
$H_{F',f}$ correspond to facets whose intersection is the ridge
\[
 \conv\bigl(\V(F\cap F')\cup\{x_0,x_1\}\bigr).
\]
Join these two vertices by their edge in $G(P^*)$. Repeating this around $C$,
including the connection from the last path back to the first,
gives a single cycle $C_f$ spanning all the blocks $H_{F,f}$
for this fixed $f$; see Figure~\ref{fig:dual-hamiltonian-construction}(b). Each spanning path in a block $H_{F,f}$
contains $t_{F,f}h_{F,f}$, so $C_f$ contains this edge
for every $F$.

\begin{figure}[htbp] 
\begin{center}     
    \includegraphics[scale=0.8]{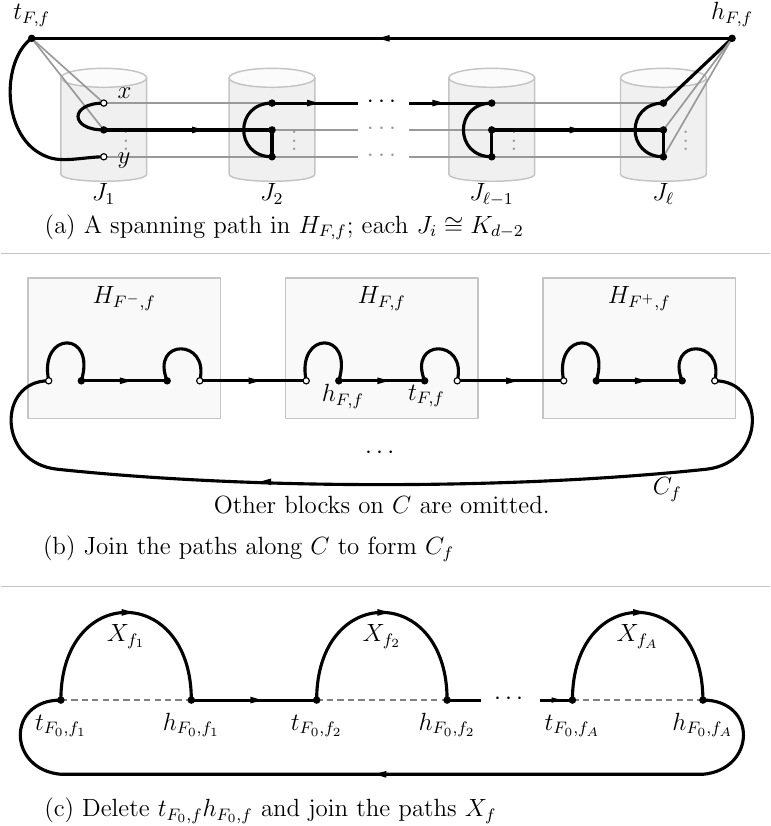}
\end{center}
\caption{Schematic construction of the hamiltonian cycle in $G(P^*)$.
(a) Cylinders represent $(d-2)$-cliques $J_i$, with matchings
between consecutive cliques. The black spanning $xy$-path contains
$t_{F,f}h_{F,f}$.
(b) Joining one such path from each block along $C$ gives $C_f$.
(c) Deleting $t_{F_0,f}h_{F_0,f}$ from each $C_f$ gives the paths $X_f$,
which are joined in the cyclic order of the edges of $D$.}
\label{fig:dual-hamiltonian-construction}
\end{figure}

Finally, fix a facet $F_0$ of $R$.
Deleting $t_{F_0,f}h_{F_0,f}$ from each cycle $C_f$ gives
a path $X_f$ through every vertex of $C_f$, with ends
$t_{F_0,f}$ and $h_{F_0,f}$.
Whenever $f,f'$ are consecutive around the oriented boundary
of $D$, we have $h(f)=t(f')$. By the definitions of
$t_{F,f}$ and $h_{F,f}$, the facets corresponding to
$h_{F_0,f}$ and $t_{F_0,f'}$ are
\[
 \conv\bigl(\V(F_0)\cup\{p_{F_0,f},h(f)\}\bigr),
 \qquad
 \conv\bigl(\V(F_0)\cup\{p_{F_0,f'},t(f')\}\bigr).
\]
Their intersection is
\(
 \conv\bigl(\V(F_0)\cup\{h(f)\}\bigr),
\)
which has $d-2+1=d-1$ vertices and is therefore a ridge of $P$.
Hence $h_{F_0,f}t_{F_0,f'}$ is an edge of $G(P^*)$.

Traverse the path $X_f$ from $t_{F_0,f}$ to $h_{F_0,f}$.
Then take the edge $h_{F_0,f}t_{F_0,f'}$, where $f'$ is
the next edge around the oriented boundary of $D$,
and continue along the path obtained from $C_{f'}$.
Following the cyclic order of the edges of $D$ joins all these
paths into a single cycle; see Figure~\ref{fig:dual-hamiltonian-construction}(c).  The facet descriptions
in~\eqref{eq:kleetope-facets} and~\eqref{eq:subdivided-facets}
show that every dual vertex belongs to exactly one block
$H_{F,f}$. Thus the paths partition $\V(G(P^*))$, and the
resulting cycle is hamiltonian.
\end{proof}

\bibliographystyle{plain}
\bibliography{references}
\end{document}